\documentclass[12pt]{amsart}
\usepackage{amsthm}
\usepackage{amsmath}
\usepackage{amssymb}

\usepackage{graphicx}

\theoremstyle{plain}
\newtheorem{thm}{Theorem}[section]
\newtheorem{prop}[thm]{Proposition}

\newtheorem{lem}[thm]{Lemma}
\newtheorem{rem}[thm]{Remark}
\newtheorem{ex}[thm]{Example}
\newtheorem{defn}[thm]{Definition}

\numberwithin{equation}{section}


\begin{document}

\title[\null]
{A Vasicek-type short rate model with memory effect}

\author[\null]{Akihiko Inoue}
\address{Department of Mathematics, Hiroshima University,
Higashi-Hiroshima Japan}

\author[\null]{Shingo Moriuchi}
\address{Bizen-Ryokuyou High School, 
Bizenshi, Japan}

\author[\null]{Yusuke Nakamura}
\address{Department of Mathematics, Hiroshima University,
Higashi-Hiroshima, Japan}

\thanks{Address correspondence to A. Inoue, Department of Mathematics, Hiroshima University,
Higashi-Hiroshima 739-8526, Japan; E-mail: inoue100@hiroshima-u.ac.jp}

\date{August 1, 2015}


\begin{abstract}
We introduce a Vasicek-type short rate model which has two additional parameters 
representing memory effect. This model 
presents better results in yield curve fitting than the classical Vasicek model. 
We derive closed-form expressions for the 
prices of bonds and bond options. Though the model is non-Markov, 
there exists an associated Markov process which allows one to apply usual numerical methods to the model. 
We derive analogs of 
an affine term structure and term structure equations for the model, 
and, using them, we present a numerical method to evaluate contingent claims.
\end{abstract}

\maketitle



\section{introduction}\label{sec:1}

The Vasicek model introduced by \cite{V} is a classical short rate model. It 
is defined by a stochastic differential equation (SDE) of the form
\begin{equation}
dr(t)=\{a-br(t)\}d t+ \sigma dW^*(t)\qquad (t\ge 0)
\label{eq:VasicekSDE326}
\end{equation}
describing the short rate process $\{r(t)\}$, where $a$, $b$ and $\sigma$ are positive 
constants and 
$\{W^*(t)\}$ is a Brownian motion under the equivalent martingale measure. 
In this paper, we introduce a Vasicek-type short rate model with memory effect 
which has some good properties.

To define the Vasicek-type model, denoted by $\mathcal{M}$, 
we replace the Brownian motion $\{W^*(t)\}$ in (\ref{eq:VasicekSDE326}) 
by a stochastic process $\{Z(t)\}$ which was introduced and investigated by \cite{AI, AIK, INA} 
(see Definition \ref{defn:Z549} below). 
The process $\{Z(t)\}$ 
is a Gaussian, stationary increment process which has parameters $p$, $q$ representing memory 
effect. 
The process 
$\{Z(t)\}$ is also an It\^o process unlike, e.g., fractional Brownian motion. 
This implies that we can apply the standard stochastic 
calculus, that is, the It\^o calculus, to the model $\mathcal{M}$.

For the Vasicek-type model $\mathcal{M}$ and $0\le t\le T$, we derive a closed-form expression for the 
time $t$ price $P(t,T)$ of a zero-coupon bond with maturity $T$, 
which we also call a {\em $T$-bond} (see Theorem \ref{thm:bondprice382} below, the proof of which 
is given in Section \ref{sec:4}). 
From the result and the forward measure method (cf.\ \cite{J} and \cite[Chapter 7]{F}), 
we also obtain a closed-form expression for the prices of 
European bond options (see Proposition \ref{prop:option274} below).

The closed-form expression for $P(t,T)$ also gives that for the yield $Y(t,T)$ defined by
\begin{equation}
Y(t,T)=-\frac{\log P(t,T)}{T-t}\qquad (t<T)
\label{eq:yield888}
\end{equation}
or $P(t,T)=\exp[-(T-t)Y(t,T)]$ (see Theorem \ref{thm:yield269} below). 
As in other short rate models, we can estimate the parameters of the Vasicek-type model $\mathcal{M}$ 
by fitting the yield curve $T\mapsto Y(0,T)$ to actual yield data. 
We will see that the model $\mathcal{M}$ shows better fitting results than the classical Vasicek model, 
thanks to the additional parameters $p$ and $q$ (see Section \ref{sec:6}).

The model $\mathcal{M}$ has memory effect. In other words, it is a non-Markov model. 
In general, non-Markov interest rate models have the drawback that numerical computations become difficult 
in them. 
It turns out that the model $\mathcal{M}$ is free from this drawback, thanks to 
the existence of an associated two-dimensional 
Markov process, which is obtained by coupling the short rate process $\{r(t)\}$ 
with another process (see (\ref{eq:SDE275}) below). 
This associated Markov process allows one to apply usual numerical methods. 
In fact, using the Markov process, we derive analogs of 
an affine term structure and term structure equations (cf.\ \cite{DK1, DK2} and \cite[Chapter 5]{F}) 
for the model $\mathcal{M}$, 
and present a numerical method to evaluate European contingent claims based on them 
(see Section \ref{sec:5}).


\section{The model}\label{sec:2}

Let $\{\mathcal{F}_t\}_{t\ge 0}$ be the $Q$-augmentation of the filtration generated by 
a one-dimensional standard Brownian motion $\{W^*(t)\}_{t\ge 0}$ defined on a complete 
probability space $(\Omega,\mathcal{F}, Q)$.

\begin{defn}\label{defn:Z549}
{\rm 
For real numbers $p$, $q$ such that
\begin{equation}
0<q<\infty, \quad -q<p<\infty,
\label{eq:pq527}
\end{equation}
we define a process $\{Z(t)\}_{t\ge 0}$ by
\begin{equation}
Z(t)=W^*(t)-\int_{0}^{t}\left\{\int_{0}^{s}pe^{-(p+q)(s-\tau)}l(\tau)dW^*(\tau)\right\}ds
\qquad (t\geq 0),
\label{eq:def-of-Z524}
\end{equation}
where the positive deterministic function $l$ is defined by
\[
l(\tau)=1-\frac{2qp}{(p+2q)^2e^{2q\tau}-p^2}\qquad (\tau\ge 0).
\]
}
\end{defn}

As is eaily seen, the process $\{Z(t)\}$ is a continuous, Gaussian, $\{\mathcal{F}_t\}$-adapted, 
It\^o process. 
It may not seem so, but $\{Z(t)\}$ is also a stationary increment process. 
In fact, if $\{\hat{W}(t)\}_{t\in\mathbb{R}}$ is another Brownian motion on 
another probability space $(\hat{\Omega}, \hat{\mathcal{F}}, \hat{Q})$, then, for the 
Gaussian, stationary increment process $\{X(t)\}_{t\ge 0}$ defined by
\begin{equation}
X(t)=\hat{W}(t) 
- \int_{0}^{t}\left\{\int_{-\infty}^s pe^{-(p+q)(s-\tau)}d\hat{W}(\tau)\right\}ds
\qquad (t\ge 0),
\label{eq:Z726}
\end{equation}
the law of $\{Z(t)\}_{t\ge 0}$ under $Q$ is equal to that of $\{X(t)\}_{t\ge 0}$ under 
$\hat{Q}$. 
See \cite[Theorem 5.2 and Example 5.3]{AIK} and \cite[Section 2]{INA}); 
(\ref{eq:Z726}) corresponds to (28) in \cite{AI} or (1.1) in \cite{INA}. 
If $p=0$, then $\{Z(t)\}$ reduces to the Brownian motion $\{W^*(t)\}$.

For $a, b, \sigma \in(0, \infty)$, we consider the following Vasicek-type SDE:
\begin{equation}
dr(t)=\{a-br(t)\}d t+ \sigma dZ(t),\quad t\ge 0, \qquad r(0)\in[0,\infty).
\label{eq:Vasicek-type-SDE541}
\end{equation}
From
\[
d\{e^{bs}r(s)\}
=e^{bs}dr(s)+be^{bs}r(s)ds \\
=ae^{bs}ds+\sigma e^{bs}dZ(s),
\]
we see that the unique strong solution to 
$(\ref{eq:Vasicek-type-SDE541})$\/ satisfies
\begin{equation}
\begin{aligned}
r(\tau)&=e^{-b(\tau-t)}r(t)+\frac{a}{b}(1-e^{-b(\tau-t)}) + \sigma\int_{t}^{\tau}e^{-b(\tau-s)}dZ(s)\\
&\qquad\qquad\qquad\qquad\qquad\qquad\qquad\qquad\qquad 
(0\le t\le \tau).
\end{aligned}
\label{eq: rep-of-r762}
\end{equation}

We consider a Vasicek-type short rate model $\mathcal{M}$ in which the short rate process 
$\{r(t)\}_{t\ge 0}$ follows (\ref{eq:Vasicek-type-SDE541}) 
or (\ref{eq: rep-of-r762}). 
We define the money-market account process $\{B(t)\}_{t\ge 0}$ by
\[
B(t):=e^{\int_{0}^{t}r(s)ds} \qquad (t\ge 0).
\]
Let $T\in (0,\infty)$. 
For $0\le t\le T$, let $P(t,T)$ be the time $t$ price of a $T$-bond. 
We regard $Q$ as an equivalent martingale measure of $\mathcal{M}$ in the sense that 
the discounted price process
\begin{equation}
\tilde{P}(t,T):=\frac{P(t,T)}{B(t)} \qquad (0\le t\le T)
\label{eq:tildeP514}
\end{equation}
becomes a $\{Q, \mathcal{F}_t\}$-martingale. 
Then we have
\begin{equation}
P(t,T)=E^Q\left.\left[e^{-\int_{t}^{T}r(s)ds}\ \right|\mathcal{F}_t\right] \qquad (0\le t\le T).
\label{eq:rep-P354}
\end{equation}
For generalities of short rate models, one can consult, e.g.,  \cite[Chapter 5]{F}.

\begin{rem}
{\rm 
If $p=0$, then $Z(t)=W^*(t)$, whence the Vasicek-type model $\mathcal{M}$ above 
reduces to the classical Vasicek model.
}
\end{rem}


\section{Prices of bonds and bond options}\label{sec:3}

Though the short rate process $\{r(t)\}$ defined by (\ref{eq:Vasicek-type-SDE541}) 
is not a Markov process, the model $\mathcal{M}$ admits a closed-form  
expression for $P(t,T)$. To state this result, 
we introduce the following function:
\[
m(t):=
\int_{0}^{t}pe^{-(p+q)s}\left\{1 - e^{-b(t-s)}\right\}ds. 
\]
We can write $m(t)$ more explicitly as
\[
m(t)=
\left\{
\begin{aligned}
&\frac{p}{p+q} + \frac{bpe^{-(p+q)t}}{(p+q-b)(p+q)} - \frac{pe^{-bt}}{p+q-b} 
\qquad (p+q-b\neq 0), \\ 
&\frac{p}{p+q} - \frac{ pe^{-(p+q)t} }{p+q} - pte^{-bt}
\qquad (p+q-b=0).
\end{aligned}
\right.
\]

Here is a closed-form expression for $P(t,T)$.

\begin{thm}\label{thm:bondprice382}
For $P(t,T)$ in $(\ref{eq:rep-P354})$ with $(\ref{eq:Vasicek-type-SDE541})$ 
and $(\ref{eq:def-of-Z524})$, we have
\[
P(t,T)=\exp\left\{-A(t,T)-C(t,T)r(t) + U(t,T)\right\} 
\qquad (0\leq t\leq T),
\]
where
\[
\begin{aligned}
&C(t,T)
:=\frac{1-e^{-b(T-t)}}{b},\\
&A(t,T):=\frac{a}{b}\left\{T-t-C(t,T)\right\} 
- \frac{\sigma^2}{2b^2}\int_{0}^{T-t} \left\{m(s) + e^{-bs} - 1\right\}^2ds\\
&\qquad\qquad\qquad\qquad\qquad\qquad\qquad\qquad 
-\frac{\sigma^2 q m^2(T-t)}{b^2 \{ (p+2q)^2 e^{2qt} -p^2 \} },\\
&U(t,T)
:=\frac{\sigma (p+2q)^2 e^{qt}m(T-t)}{b\{(p+2q)^2 e^{2qt} -p^2\}}
\int_{0}^{t} \left( e^{qs}-\frac{p}{p+2q} e^{-qs} \right)dZ(s).
\end{aligned}
\]
\end{thm}

We prove Theorem \ref{thm:bondprice382} in Section \ref{sec:4}.

Recall $Y(t,T)$ from (\ref{eq:yield888}). 
From Theorem \ref{thm:bondprice382}, we immediately obtain the next theorem.

\begin{thm}\label{thm:yield269}
For $Y(t,T)$ in $(\ref{eq:yield888})$ with $P(t,T)$ as in Theorem \ref{thm:bondprice382}, 
we have
\[
Y(t,T)=\frac{A(t,T)}{T-t} + \frac{C(t,T)}{T-t}r(t) - \frac{U(t,T)}{T-t}\qquad (0\le t<T).
\]
In particular,
\begin{equation}
Y(0,T)=\frac{A(0,T)}{T} + \frac{C(0,T)}{T}r(0)\qquad (T>0).
\label{eq:yield589}
\end{equation}
\end{thm}

In Section \ref{sec:6}, we use (\ref{eq:yield589}) to estimate the parameters of the model 
$\mathcal{M}$ from actual yield data.

Recall $\tilde{P}(t,T)$ from (\ref{eq:tildeP514}). 

\begin{prop}\label{prop:PtildeSDE777}
The process $\{\tilde{P}(t,T)\}_{0\le t\le T}$ satisfies 
the SDE
\begin{equation}
d\tilde{P}(t,T)=v(t,T)\tilde{P}(t,T)dW^*(t)\qquad (0\le t\le T),
\label{eq:PtildeSDE479}
\end{equation}
where the deterministic function $v(t,T)$ is defined by
\begin{equation}
v(t,T)
:=\frac{\sigma}{b}\left\{e^{-b(T-t)}-1
+ l(t)m(T-t)\right\} \qquad (0\le t\le T).
\label{eq:Pvola296}
\end{equation}
\end{prop}

\begin{proof}
For fixed $T>0$, we put
\[
V(t):=-A(t,T)-C(t,T)r(t)+U(t,T)-\int_{0}^{t} r(u) du \qquad (0\le t\le T),
\]
so that $\tilde{P}(t,T)=\exp\{V(t)\}$. Then, from (\ref{eq:Vasicek-type-SDE541}) and 
(\ref{eq:def-of-Z524}), we have
\[
\begin{aligned}
&dV(t) = \Phi(t) dt \\
&\quad + \sigma\left\{ -C(t,T) 
+ \frac{(p+2q)^2 e^{qt}m(T-t)}{b((p+2q)^2 e^{2qt} -p^2)} 
\left( e^{qt} -\frac{p}{p+2q} e^{-qt} \right) \right\} dW^*(t) \\
&\quad\ \ \ \ = \beta(t) dt + v(t,T)dW^*(t)
\end{aligned}
\]
for some continuous $\{\mathcal{F}_t\}$-adapted process $\{\beta(t)\}_{0\le t\le T}$. 
This and It\^o's formula yield
\[
d\tilde{P}(t,T)=v(t,T)\tilde{P}(t,T)dW^*(t) + \gamma(t)dt\qquad (0\le t\le T)
\]
for another continuous $\{\mathcal{F}_t\}$-adapted process $\{\gamma(t)\}_{0\le t\le T}$. 
However, since $\{\tilde{P}(t,T)\}_{0\le t\le T}$ is a martingale,  we see that $\gamma(t)\equiv 0$. 
Thus the proposition follows.
\end{proof}

Let $\pi_C(0)$ be the time $0$ price of a European call option on a $T$-bond, with maturity $S\in (0,T)$ 
and exercise price $K\in (0,\infty)$. Then $\pi_C(0)$ is given by
\begin{equation}
\pi_C(0)=E^Q\left[e^{-\int_0^Sr(s)ds}(P(S,T) - K)_+\right],
\label{eq:call-price529}
\end{equation}
where $(x)_+:=\max(x,0)$. We write $\Phi$ for the cumulative distribution function of a standard normal distribution:
\[
\Phi(x):=\int_{-\infty}^x \frac{1}{\sqrt{2\pi}}e^{-(1/2)y^2}dy.
\]

\begin{prop}\label{prop:option274}
For $\pi_C(0)$ in $(\ref{eq:call-price529})$ with $\{r(t)\}$ and $\{P(t,T)\}$ as in Theorem \ref{thm:bondprice382}, 
we have
\[
\pi_C(0)=P(0,T)\Phi(d_+) - KP(0,S)\Phi(d_-),
\]
where
\begin{align*}
&d_+:=\frac{1}{\sqrt{\Sigma^2(S)}}\left\{\log\left(\frac{P(0,T)}{KP(0,S)}\right) + \frac{1}{2}\Sigma^2(S)\right\},\\
&d_-:=d_+ - \sqrt{\Sigma^2(S)},\qquad \Sigma^2(S):=\int_0^S v_{S,T}^2(t)dt,\\
&v_{S,T}(t):=\frac{\sigma}{b}\left\{e^{-b(T-t)}-e^{-b(S-t)}
+ l(t)(m(T-t) - m(S-t))\right\}.
\end{align*}
\end{prop}

\begin{proof}
Notice that $v_{S,T}(t)=v(t,T) - v(t,S)$. 
Since $v(t,T)$ is a deterministic function, 
the proposition follows immediately from Proposition \ref{prop:PtildeSDE777} and \cite[Proposition 7.2]{F}.
\end{proof}

We use Proposition \ref{prop:option274} to illustrate the efficiency of a numerical method in 
Section \ref{sec:5}.

\begin{rem}
The Vasicek model $(\ref{eq:VasicekSDE326})$ is one of the affine term-structure models which are 
extensively used in finance. The CIR model is another example. See \cite{DK1, DK2, F}. See also 
\cite{BSS, FS, ALS} for related SDEs and diffusion processes. It is an open problem to obtain 
results similar to those in this paper for the CIR-type and other models that are obtained by replacing 
the Brownian motion by the process $\{Z(t)\}$.
\end{rem}


\section{Proof of Theorem \ref{thm:bondprice382}}
\label{sec:4}

From (\ref{eq: rep-of-r762}),
\[
\begin{aligned}
\int_{t}^{T}r(\tau)d\tau
&=\sigma\int_{t}^{T}\left\{ \int_{t}^{\tau}e^{-b(\tau-s)}dZ(s)\right\}d\tau\\
&\qquad\qquad\qquad   + C(t,T)r(t) + \frac{a}{b}\left\{T-t-C(t,T)\right\}.
\end{aligned}
\]
By the classical Fubini theorem as well as that for stochastic integrals 
(cf.\ \cite[Section 6.5]{F}), the first term on the right-hand side is equal to
\[
\sigma\int_{t}^{T}\frac{1 - e^{-b(T-s)}}{b}dZ(s)
=-\sigma \int_t^{\infty}f(s)dZ(s),
\]
where
\[
f(s):=\frac{\{e^{-b(T-s)}-1\}}{b}I_{[t,T]}(s).
\]
Combining these with 
(\ref{eq:rep-P354}), we obtain
\begin{equation}
\begin{aligned}
P(t,T)&= \exp \left\{ - \frac{a}{b} \left( T-t-C(t,T)\right) - C(t,T)r(t) \right\}\\
&\qquad\qquad\qquad\quad 
\times E^Q\left.\left[\exp\left\{\sigma\int_{t}^{\infty} f(s)dZ(s)\right\}\right|\mathcal{F}_t\right].
\end{aligned}
\label{eq:1st-step528}
\end{equation}
We write $P_{[0,t]}$ for the orthogonal projection from $L^2(\Omega,\mathcal{F},Q)$ onto 
the closed subspace spanned by $\{Z(\tau): \tau\in [0,t]\}$, and $P_{[0,t]}^\bot$ for 
its orthogonal complement: 
$P^\bot_{[0,t]}X:=X-P_{[0,t]}X$ for $X\in L^2(\Omega,\mathcal{F},Q)$. 
From \cite[Theorem 5.2 and Example 5.3]{AIK} and \cite[Section 2]{INA}, we have, for $t\ge 0$,
\[
W^*(t)=Z(t) + \int_0^t\left\{\int_0^s
p(p+2q)\frac{(p+2q)e^{q\tau} - pe^{-q\tau}}{(p+2q)e^{q\tau} - p^2e^{-q\tau}}
dZ(\tau)
\right\}ds.
\]
This and (\ref{eq:def-of-Z524}) imply
\begin{equation}
\sigma(Z(s): s\in [0,t]) = \sigma(W^*(s): s\in [0,t]) \qquad (t\ge 0).
\label{eq: innovation418}
\end{equation}
Since the process $\{Z(s)\}$ is Gaussian, we see from (\ref{eq: innovation418}) that $\mathcal{F}_t$ 
and $P_{[0,t]}^\bot\int_{t}^{\infty}f(s)dZ(s)$ are independent.
Hence,
\begin{equation}
\begin{aligned}
&E^Q\left.\left[\exp\left\{\sigma\int_{t}^{\infty} f(s)dZ(s)\right\}\right|\mathcal{F}_t\right]
 = \exp\left\{\sigma P_{[0,t]}\int_{t}^{\infty}f(s)dZ(s)\right\}
\\
&\qquad\qquad\qquad\qquad\qquad\quad \times 
\exp\left\{\frac{\sigma^2}{2}\left\Vert P_{[0,t]}^\bot\int_{t}^{\infty}f(s)dZ(s)\right\Vert^2\right\},
\end{aligned}
\label{eq:2nd-step528}
\end{equation}
where $\Vert X\Vert:=E^Q[X^2]^{1/2}$ for $X \in L^2(\Omega, \mathcal{F}, Q)$.

First, we calculate $P_{[0,t]}\int_{t}^{\infty}f(s)dZ(s)$. 
From \cite[Theorem 4.7 and Example 5.3]{AIK}, 
\[
P_{[0,t]}\int_{t}^{\infty}f(s)dZ(s)
=\int_{0}^{t}\left\{\int_{0}^{\infty}g(s,\tau;0,t)f(t+\tau)d\tau \right\}dZ(s),
\]
where
\begin{align*}
&\phi(t):=-\frac{pe^{-qt}}{p+2q}, \label{eq: def-of-phi246}\\
&b(s,\tau):=-pe^{-qs}e^{-(p+q)\tau},\\
&g(s,\tau ; 0,t):=b(t-s,\tau)\frac{1}{1-\phi^2(t)}+b(s,\tau)\frac{\phi(t)}{1-\phi^2(t)}.
\end{align*}
However, since $I_{[t,T]}(t+\tau)=I_{[0,T-t]}(\tau)$, we have
\[
\begin{aligned}
&\int_{0}^{\infty}g(s,\tau;0,t)f(t+\tau)d\tau \\
&\quad =\frac{p\left\{e^{-q(t-s)}+\phi(t)e^{-qs}\right\}}{b\{1-\phi^2(t)\}}
\int_{0}^{T-t}e^{-(p+q)\tau}\left\{e^{-b(T-t-\tau)}-1\right\}d\tau\\
&\quad = - \frac{e^{-qt}m(T-t)}{b\{1-\phi^2(t)\}}  \left( e^{qs} -\frac{p}{p+2q}e^{-qs} \right).
\end{aligned}
\]
Thus
\begin{equation}
P_{[0,t]}\int_{t}^{\infty}f(s)dZ(s)= -\frac{1}{\sigma}U(t,T).
\label{eq:3rd-step831}
\end{equation}

Next, we calculate $\Vert P_{[0,t]}^\bot\int_{t}^{\infty}f(s)dZ(s)\Vert^2$. 
By \cite[Theorem 4.12 and Example 5.3]{AIK},
\[
\left\Vert P_{[0,t]}^\bot\int_{t}^{\infty}f(s)dZ(s)\right\Vert^2
=\sum^\infty_{n=0}\int_{0}^{\infty}d_n^2(s,f)ds,
\]
where
\[
d_0 (s,f):=-\int_{0}^{\infty}c(u)f(t+s+u)du+f(t+s)\qquad(s>0)
\]
and
\begin{equation}
\begin{aligned}
&d_n (s,f):=-\int_{0}^{\infty}c(u)\left\{\int_{0}^{\infty}b_n (t+u+s,\tau)f(t+\tau)d\tau\right\}du \\
&\qquad\qquad\qquad +\int_{0}^{\infty}b_n (t+s,\tau)f(t+\tau)d\tau\qquad(n\in\mathbb{N},\ s>0)
\end{aligned}
\label{eq:Dn852}
\end{equation}
with
\[
\begin{aligned}
&c(s):=pe^{-(p+q)s}I_{(0,\infty)}(s),\\
&b_n(s,\tau):=-\phi^{n-1}(t)pe^{-qs}e^{-(p+q)\tau}\qquad (n\in\mathbb{N},\ s,\tau>0).
\end{aligned}
\]
Since $I_{[t,T]}(t+s+u)=I_{[0,T-t-s]}(u)I_{[0,T-t]}(s)$, we have
\[
\begin{aligned}
&\int_{0}^{\infty}c(u)f(t+s+u)du \\
&\quad =\frac{p}{b}I_{[0,T-t]}(s)\int_{0}^{T-t-s}e^{-(p+q)u}\left\{e^{-b(T-t-s-u)}-1\right\}du \\
&\quad = - \frac{1}{b}I_{[0,T-t]}(s)m(T-t-s),
\end{aligned}
\] 
so that
\[
d_0(s,f)=I_{[0,T-t]}(s)\left\{\frac{1}{b} m(T-t-s)+f(t+s)\right\}.
\]
Thus
\begin{equation}
\begin{aligned}
\int_{0}^{\infty}d_0^2(s,f)ds
&=\int_{0}^{T-t}\left\{\frac{1}{b}m(T-t-s) + f(t+s)\right\}^2ds\\
&=\frac{1}{b^2}\int_{0}^{T-t}\left\{m(s) + e^{-bs} - 1\right\}^2ds.
\end{aligned}
\label{eq:intd0123}
\end{equation}
For $n\in\mathbb{N}$ and $\tau\in [0,T-t]$, we have
\[
\int_{0}^{\infty}b_n (t+u+s,\tau)f(t+\tau)d\tau
= \frac{1}{b}\phi^{n-1}(t)e^{-q(t+s+u)}m(T-t).
\]
So, the first term on the right-hand side of (\ref{eq:Dn852}) is equal to
\[
- \frac{pe^{-q(t+s)}}{b(p+2q)}m(T-t)\phi^{n-1}(t),
\]
while the second term to $-(1/b)e^{-q(t+s)}m(T-t)\phi^{n-1}(t)$, whence
\[
d_n(s,f)
= \frac{2qe^{-q(t+s)}}{b(p+2q)}m(T-t)\phi^{n-1}(t).
\]
Therefore,
\[
\int_{0}^{\infty} d_n^2(s,f)ds=\frac{2qe^{-2qt}}{b^2(p+2q)^2}m^2(T-t)\phi^{2(n-1)}(t).
\]
From this and $|\phi(t)|<1$, we get
\begin{equation}
\begin{aligned}
\sum^\infty_{n=1}\int_{0}^{\infty}d_n^2(s,f)ds
&=\frac{2qe^{-2qt}m^2(T-t)}{b^2(p+2q)^2\{1 - \phi^{2}(t)\}} \\
&=\frac{2 q m^2(T-t)}{ b^2\{(p+2q)^2 e^{2qt} -p^2\}}.
\end{aligned}
\label{eq:4th-step628}
\end{equation}
Theorem \ref{thm:bondprice382} follows from 
(\ref{eq:1st-step528}), 
(\ref{eq:2nd-step528}), 
(\ref{eq:3rd-step831}), 
(\ref{eq:intd0123}), 
and (\ref{eq:4th-step628}).

\begin{rem}
If we use other functions as $\l$ in $(\ref{eq:def-of-Z524})$, then 
the above proof of Theorem \ref{thm:bondprice382} does not work. 
For, Theorems 4.7 and 4.12 in \cite{AIK} are no more available. 
The authors do not know other functions which give results similar to those 
in Section \ref{sec:3}.
\end{rem}


\section{An associated Markov process}\label{sec:5}

We define
\[
u(t)
:= \int_{0}^{t} e^{(p+q)s} l(s) dW^*(s) \qquad (t\ge 0)
\]
so that
\[
Z(t)=W^*(t) - \int_0^tpe^{-(p+q)s}u(s)ds\qquad (t\ge 0),
\label{eq:Z-u328}
\]
and couple the short rate process $\{r(t)\}$ with $\{u(t)\}$. 
Then, the resulting two-dimensional process $\{(r(t), u(t))^{\mathrm{T}}\}$ satisfies the SDE
\begin{equation}
\begin{aligned}
d\left( 
\begin{array}{cc}
r(t) \\
u(t)
\end{array}\right) &=
\left\{
\left(
\begin{array}{cc}
a\\
0
\end{array}
\right)
-
\left(
\begin{matrix}
b & \sigma pe^{-(p+q)t} \\
0 & 0
\end{matrix}
\right)
\left(
\begin{array}{cc}
r(t) \\
u(t)
\end{array}
\right)
\right\} dt 
\\
&\qquad\qquad +
\left( \begin{array}{cc}
\sigma \\
e^{(p+q)t} l(t)
\end{array}\right) dW^*(t)\qquad (t\ge 0),
\end{aligned}
\label{eq:SDE275}
\end{equation}
whence it is a Markov process. This implies that 
$P(t,T)$, which is defined by (\ref{eq:rep-P354}), would have an expression of 
the form $P(t,T)=F(t,r(t),u(t); T)$. The aim of this section is to derive an explicit 
expression for such $F$ and to present a numerical method to evaluate contingent claims based on it.

\begin{lem}\label{lem:Zint384}
We have
\[
\int_{0}^{t} \left( e^{qs} -\frac{p}{p+2q} e^{-qs} \right) dZ(s) 
= \frac{e^{-pt}}{l(t)} \left( 1-\frac{p e^{-2qt}}{p+2q}  \right)u(t)\qquad (t\ge 0).
\]
\end{lem}

\begin{proof}
Since $u(0)=0$ and $e^{qt}dZ(t)=\{e^{-pt}/l(t)\}du(t) - pe^{-qt}u(t)dt$, 
\[
\int_{0}^{t}  e^{qs} dZ(s) 
=\frac{e^{-pt}}{l(t)} u(t) -\int_{0}^{t} \left\{ \frac{e^{-ps}}{l(s)}  \right\}' u(s) ds 
-  \int_{0}^{t}  pe^{-ps}  u(s)ds.
\]
Similarly,
\[
\begin{aligned}
\int_{0}^{t} e^{-qs} dZ(s) 
&= \frac{e^{-(p+2q)t}}{l(t)} u(t) -\int_{0}^{t} \left\{ \frac{e^{-(p+2q)s}}{l(s)}  \right\}' u(s) ds \\
&\qquad\qquad\qquad\qquad\quad -  \int_{0}^{t}  pe^{-(p+2q)s}  u(s)ds.
\end{aligned}
\]
Combining,
\[
\int_{0}^{t} \left( e^{qs} -\frac{p}{p+2q} e^{-qs} \right) dZ(s) 
= \frac{e^{-pt}}{l(t)} \left( 1-\frac{p e^{-2qt}}{p+2q}  \right)u(t) + \int_{0}^{t} R(s) u(s) ds,
\]
where
\[
R(s):= \frac{p}{p+2q} \left\{ \left( \frac{ e^{-(p+2q)s} }{l(s)} \right)' 
 + pe^{-(p+2q)s}  \right\} - \left\{ \frac{e^{-ps}}{l(s)} \right\}' - pe^{-ps}.
\]  
However, from
\[
R(s)=\frac{pe^{-ps} - pe^{-(p+2q)s}}{l(s)} 
+  \left\{  \frac{p e^{-(p+2q)s}}{p+2q} -e^{-ps} \right\} 
\left\{\left(\frac{1}{l(s)} \right)' + p\right\}
\]
and
\[
\left\{  \frac{p e^{-(p+2q)s}}{p+2q} -e^{-ps}\right\} \left\{ \frac{1}{l(s)}  \right\}'
=2q e^{-ps} \left\{ \frac{1}{l(s)} - 1   \right\},
\]
we see that $R(s)\equiv 0$. Thus the lemma follows.
\end{proof}

Here is the explicit expression for $F$ stated above. It may be regarded as 
an {\em affine term-structure}\/ for the model (cf.\ \cite{DK1, DK2} and \cite[Section 5.3]{F}).

\begin{thm}\label{thm:54}
Let $P(t,T)$ be as in Theorem \ref{thm:bondprice382}. 
Then we have
\[
P(t,T)=F(t,r(t),u(t);T)\qquad (0\leq t\leq T),
\]
where
\begin{equation}
\begin{aligned}
F(t,x,y;T)&:=\exp\left\{-A(t,T)-C(t,T)x + D(t,T)y\right\} \\
&\qquad\qquad\qquad\qquad 
((t,x,y)\in [0,T]\times \mathbb{R}\times \mathbb{R})
\end{aligned}
\label{eq:AFT628}
\end{equation}
with deterministic function $D(t,T)$ defined by
\[
D(t,T):=\frac{\sigma }{b} e^{-(p+q)t} m(T-t). \\
\]
\end{thm}

\begin{proof}
By Lemma \ref{lem:Zint384}, we have
\[
U(t,T)=\frac{\sigma }{b} e^{-(p+q)t} m(T-t) u(t)=D(t,T) u(t).
\]
This and Theorem \ref{thm:bondprice382} yield the theorem.
\end{proof}

We turn to analogs of term-structure equations for 
European contingent claims in the model $\mathcal{M}$.

\begin{prop}\label{prop:PDE537}
Let $S\in (0,\infty)$ and 
let $g:\mathbb{R}\times\mathbb{R} \to \mathbb{R}$ and 
$G:[0,S]\times \mathbb{R}\times\mathbb{R}\to\mathbb{R}$ 
be continuous. 
We assume that $G$ 
is in $C^{1,2,2}([0,S)\times \mathbb{R}\times\mathbb{R})$ and satisfies
\begin{equation}
\left\{
\begin{aligned}
&\frac{\partial G}{\partial t}(t,x,y) + \mathcal{L}G(t,x,y) = 0\qquad ((t,x,y)\in [0,S)\times 
\mathbb{R}\times \mathbb{R}),\\
& G(S,x,y)=g(x,y) \qquad ((x,y)\in \mathbb{R}\times \mathbb{R}),
\end{aligned}
\right.
\label{eq:PDE263}
\end{equation}
where
\[
\begin{aligned}
\mathcal{L}G
&:=\frac{\sigma^{2}}{2} \frac{\partial^2 G}{\partial x^2} 
+ \sigma e^{(p+q)t} l(t) \frac{\partial^2 G}{\partial x\partial y} 
+ \frac{\{e^{(p+q)t} l(t)\}^2}{2} \frac{\partial^2 G}{\partial y^2} \\
&\qquad\qquad\qquad\qquad\qquad\qquad 
+ \left\{a-bx- p \sigma e^{-(p+q)t} y\right\} \frac{\partial G}{\partial x} -x G.
\end{aligned}
\]
We also assume
\begin{equation}
\int_{0}^{S} E^{Q}\left[e^{-2 \int_{0}^{t} r(s) ds}
\left\{\Delta G(t,r(t),u(t))\right\}^2\right] dt< \infty,
\label{eq:L2condtion624}
\end{equation}
where
\[
\Delta G:=\sigma \frac{\partial G}{\partial x} + e^{(p+q)t} l(t) \frac{\partial G}{\partial y}.
\]
Then
\begin{equation}
E^{Q}\left[\left. e^{-\int_{t}^{S} r(s) ds} g(r(S),u(S))\right|\mathcal{F}_{t}\right]
=G(t,r(t),u(t)) \quad 
(0\le t\le S).
\label{eq:price274}
\end{equation}
\end{prop}

\begin{proof}
From (\ref{eq:SDE275}), (\ref{eq:PDE263}) and It\^o's formula, we have 
\[
dG(t,r(t),u(t))
=r(t)G(t,r(t),u(t))dt + \Delta G(t,r(t), u(t)))dW^*(t),
\]
whence
\[
\begin{aligned}
d\left\{ e^{-\int_{0}^{t} r(s) ds} G(t,r(t),u(t)) \right\}
=e^{-\int_{0}^{t} r(s) ds} \Delta G(t,r(t),u(t))dW^*(t).
\end{aligned} 
\]
From this and (\ref{eq:L2condtion624}), we see that 
$\{e^{-\int_{0}^{t} r(s) ds} G(t,r(t),u(t))\}_{0\leq t\leq S}$ is a square integrable martingale, 
whence
\[
\begin{aligned}
e^{-\int_{0}^{t} r(s) ds}G(t,r(t),u(t))
&=E^{Q}\left[\left. e^{-\int_{0}^{S} r(s) ds} G(S,r(S),u(S))\right|\mathcal{F}_{t}\right]\\
&=E^{Q}\left[\left. e^{-\int_{0}^{S} r(s) ds} g(r(S),u(S))\right|\mathcal{F}_{t}\right].
\end{aligned}
\]
Thus (\ref{eq:price274}) follows.
\end{proof}

Suppose that, for $0<S\le T$ and a continuous function $h: (0,\infty)\to \mathbb{R}$, 
we want to evaluate a European contingent claim 
$h(P(S,T))$ with maturity $S$, written on a $T$-bond. 
Since
\[
h(P(S,T)) = h(F(S,r(S),u(S);T))
\]
for $F$ in (\ref{eq:AFT628}), 
the price of the claim at time $t\in [0,S]$ is given by the left-hand side of (\ref{eq:price274}) 
with $g(x,y)=h(F(S,x,y;T))$. Therefore, by Proposition \ref{prop:PDE537} and under the assumption that 
(\ref{eq:PDE263}) has a suitable solution, the evaluation is reduced to solving 
(\ref{eq:PDE263}) numerically. In this sense, we may regard (\ref{eq:PDE263}) as a 
{\em term-structure equation\/} for the claim $h(P(S,T))$ in the model $\mathcal{M}$ 
(cf.\ \cite[Section 5.2]{F}).

We illustrate the efficiency of the numerical method stated above 
with two examples.

\begin{ex}
{\rm 
We consider a $T$-bond. Here we take the values
\[
T=1,\quad a=0.12,\quad b=1.9,\quad \sigma=0.35,\quad p=0.034,\quad q=0.12.
\]
In Figure 1, 
we compare the values of $G(0,r(0),0)$ obtained solving (\ref{eq:PDE263}) 
with $g(x,y)\equiv 1$ numerically and the exact values given by $P(0,1)=F(0,r(0),0;1)$ with 
$F$ in (\ref{eq:AFT628}). 
}
\end{ex}

\begin{ex}
{\rm 
We consider a bond call option $(P(S,T)-K)_+$ with maturity $S$ and 
exercise price $K$. We take the values
\begin{align*}
&T=1,\quad K=0.3,\quad a=0.08,\quad b=1.5,\quad \sigma=0.3,\\
&p=0.07,\quad q=0.08,
\quad r(0)=0.025\ \ (\mbox{or $r(0)=2.5$ \%}).
\end{align*}
In Figure 2, 
we compare the values of $G(0,r(0),0)$ obtained solving (\ref{eq:PDE263}) 
with $g(x,y)=(F(S,x,y;T) - K)_+$ and $F$ in (\ref{eq:AFT628}) numerically and 
the exact values given by Proposition \ref{prop:option274}. 
}
\end{ex}

\begin{figure}
\begin{center}
\includegraphics[width=12cm]{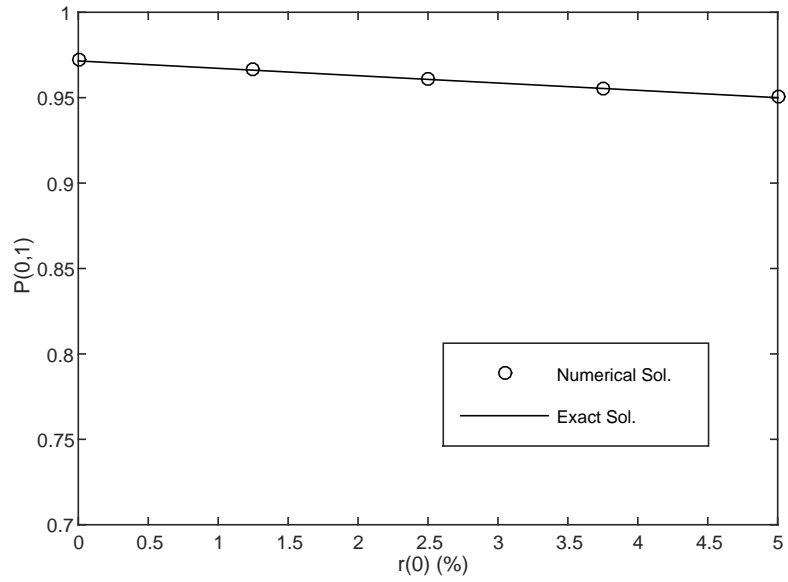}
\caption{Zero-Coupon Bond Prices}
\end{center}
\end{figure}

\begin{figure}
\begin{center}
\includegraphics[width=12cm]{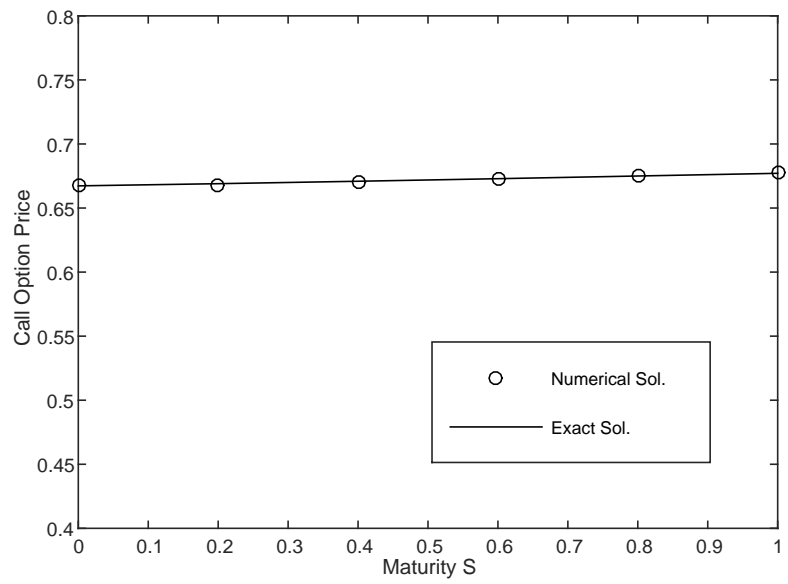}
\caption{Call Option Prices}
\end{center}
\end{figure}


\section{Estimation of parameters}\label{sec:6}
We estimate the parameters of $\mathcal{M}$ by fitting the yield curve 
to actual yield data. More precisely, we search for the values of parameters such that the 
following least squares error is minimized:
\[
\sum_{i=1}^{n} \left\{y(0,T_i)-Y(0,T_i)\right\}^2
\]
(see \cite[Chapter 3]{F} and \cite[Section 30.8]{H} for related discussions on calibrating term-structure models). 
Here, $Y(0,T)$ is the yield for a $T$-bond in the model $\mathcal{M}$ given by (\ref{eq:yield589}), 
and $y(0,T_i)$'s are actual yield data. We take $n=10$ and
\begin{align*}
&T_{1}=\frac{1}{12},\quad T_{2}=\frac{3}{12},\quad T_{3}=\frac{6}{12},\quad T_{4}=1,\quad T_{5}=2,\\
&T_{6}=3,\quad T_{7}=5,\quad T_{8}=7,\quad T_{9}=10,\quad T_{10}=20\quad \mbox{(in years)}.
\end{align*}
As $y(0,T_i)$'s, we use the Daily Treasury Yield Curve Rates published at the US Treasury Website. 

First, we use the data on December 31, 2007. The estimated values of the parameters are as follows:
\begin{align*}
&a=0.1635,\quad b=1.8952,\quad \sigma=0.7247,\quad p=0.0909,\quad q=0.2100,\\
&r(0)=0.0240\ \ (\mbox{or $r(0)=2.40$ \%}).
\end{align*}
In Figure 3, we show the data $y(0,T_i)$ and the fitted yield curve $Y(0,T)$. 
For comparison, we also show the Vasicek yield curve fitted to the data. 

Next, we use the data on May 24, 2005. The estimated values of the parameters are as follows:
\begin{align*}
&a=0.0822,\quad b=1.5561,\quad \sigma=0.3007,\quad p=0.0696,\quad q=0.0758,\\
&r(0)=0.0259\ \ (\mbox{or $r(0)=2.59$ \%}).
\end{align*}
The data and fitted yield curves are shown in Figure 4.

Finally, we use the data on September 15, 2008. The estimated values of the parameters are as follows:
\begin{align*}
&a=0.1216,\quad b=1.6806,\quad \sigma=0.6246,\quad p=0.1170,\quad q=0.1623,\\
&r(0)=1.0010 \times 10^{-5}\ \ (\mbox{or $r(0)=1.0010 \times 10^{-3}$ \%}).
\end{align*}
The data and fitted yield curves are shown in Figure 5.

\begin{figure}
\begin{center}
\includegraphics[width=12cm]{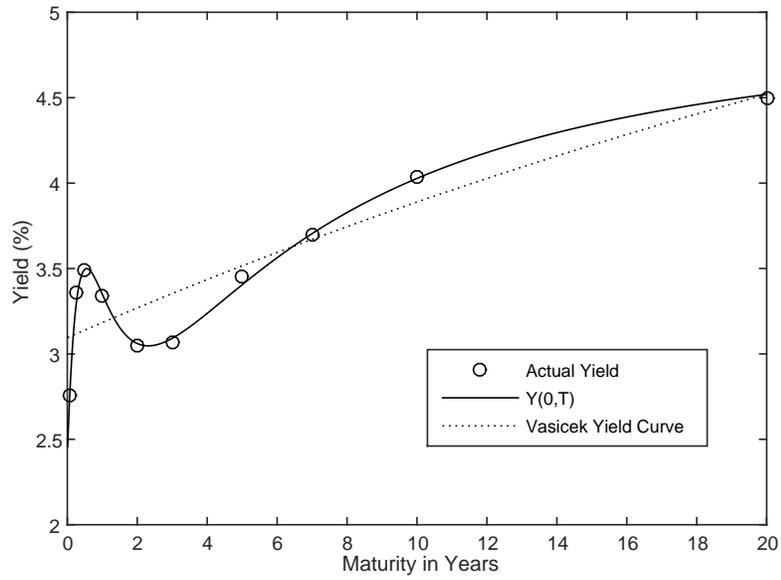}
\caption{Yield Curve on December 31, 2007}
\end{center}
\end{figure}

\begin{figure}
\begin{center}
\includegraphics[width=12cm]{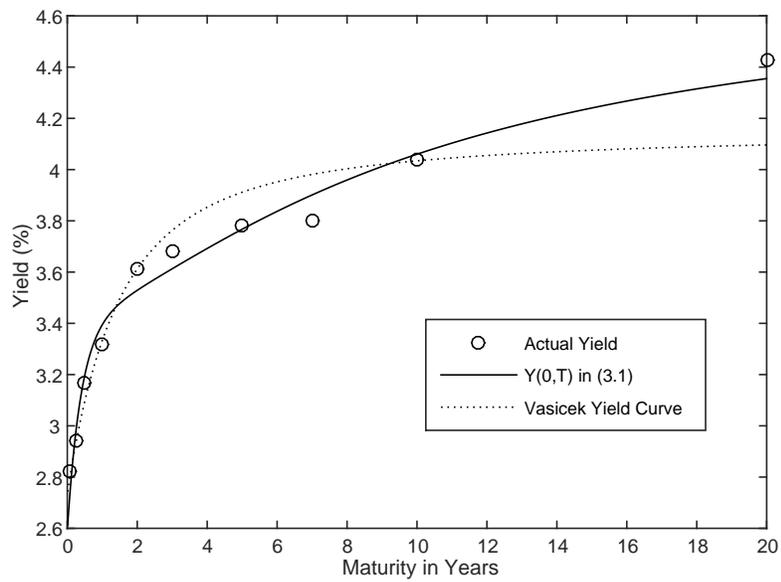}
\caption{Yield Curve on May 24, 2005}
\end{center}
\end{figure}

\begin{figure}
\begin{center}
\includegraphics[width=12cm]{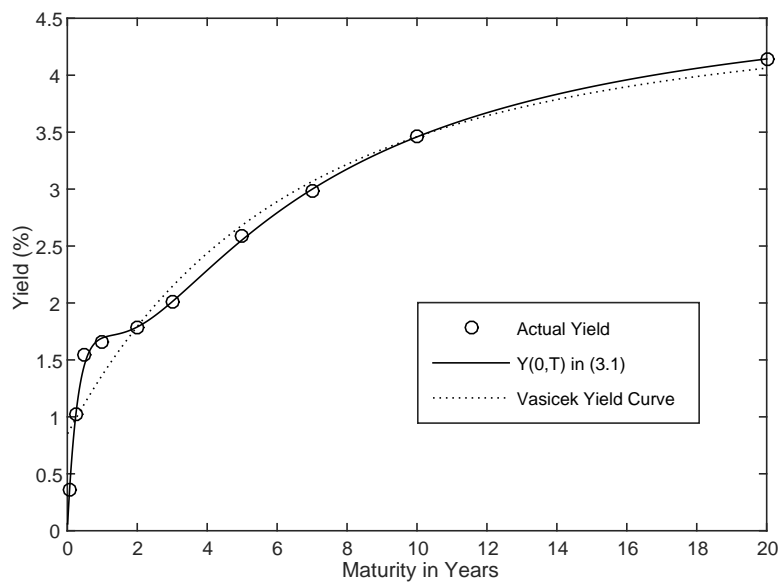}
\caption{Yield Curve on September 15, 2008}
\end{center}
\end{figure}


\end{document}